\documentclass[12pt,oneside]{amsart}

\usepackage{amssymb}
\usepackage{amsmath}
\usepackage{amsthm}

\usepackage{longtable}

\theoremstyle{plain}
\newtheorem{theorem}{Theorem}
\newtheorem{lemma}{Lemma}
\newtheorem{corollary}{Corollary}
\newtheorem{proposition}{Proposition}

\theoremstyle{definition}
\newtheorem{definition}{Definition}
\newtheorem{example}{Example}

\theoremstyle{remark}
\newtheorem{remark}{Remark}

\begin{document}
\title
[A convex level set and the Monge-Amp\`ere current]
{On a convex level set of a plurisubharmonic function and the support of the Monge-Amp\`ere current}
\author{Yusaku Tiba}
\date{}
\maketitle
\begin{abstract}
In this paper, we study a geometric property of a continuous plurisubharmonic function which is 
a solution of the Monge-Amp\`ere equation and has a convex level set.  
To prove our main theorem, we show a minimum principle of a maximal plurisubharmonic function.  
By using our results and Lempert's results, 
we show a relation between the supports of the Monge-Amp\`ere currents and 
complex $k$-extreme points of closed balls for the Kobayashi distance in a bounded convex domain in 
$\mathbb{C}^{n}$.  
\end{abstract}

%\classification{32U35, 32W20}

\section{Introduction}

Let $D$ be a domain in $\mathbb{C}^{n}$.  
Let $\mathrm{Psh}(D)$ be plurisubharmonic functions in $D$.  
In \cite{BT76} and \cite{BT82}, the Monge-Amp\`ere current $(dd^{c}u)^{k}$ ($k=1, \ldots, n$) is defined 
for $u \in \mathrm{Psh}(D) \cap C^{0}(D)$.  
The function $u \in \mathrm{Psh}(D) \cap C^{0}(D)$ such that $(dd^{c}u)^{n} = 0$ in $D$ can be  characterized as the maximal plurisubharmonic function (c.f.~\cite{Kli}).  
Here we say $u \in \mathrm{Psh}(D) \cap C^{0}(D)$ is maximal if 
for every relatively compact open subset $G$ of $D$, and for each upper semicontinuous function $v$ 
on $\overline{G}$ such that $v \in \mathrm{Psh}(G)$ and $v \leq u$ on $\partial G$, we have $v \leq u$ in $G$.  
If $u \in \mathrm{Psh}(D) \cap C^{3}(D)$ such that $(dd^{c}u)^{k} = 0$ in $D$, 
there exists a foliation on the interior of the support of $(dd^{c}u)^{k-1}$ by complex ($n-k+1$)-dimensional submanifolds such that the restriction of $u$ 
to any leaf of the foliation is pluriharmonic (see \cite{BK}).  
The geometric properties of a solution of $(dd^{c}u)^{k} = 0$ are more complicated 
in lower regularity.  
For example, Sibony showed that there exist a $C^{1, 1}$ plurisubharmonic function $u$ in the 
Euclidean unit ball $B$ in $\mathbb{C}^{2}$ and a point $p \in B$ such that $(dd^{c}u)^{2} = 0$ 
in $B$, $u$ attains its minimum at $p$ and there exists no holomorphic disk through $p$ on which $u$ is harmonic
(see \cite{Duj} for Sibony's example and more examples).  

In this paper, we study a geometric property of $u \in \mathrm{Psh}(D) \cap C^{0}(D)$ which is a solution of the 
Monge-Amp\`ere equation 
$(dd^{c}u)^{k} = 0$ ($k = 1, \ldots n-1$) and has a convex 
level set.  

Let 
\begin{align*} 
& \overline{B}_{u}(t) = \{z \in D; u(z) \leq t\},\\
& S_{u}(t) = \{z \in D; u(z)=t\}
\end{align*}
for $t \in \mathbb{R}$.  
Our main result is the following.  

\begin{theorem}\label{theorem:a}
Let $D$ be a domain in $\mathbb{C}^{n}$ and $u \in \mathrm{Psh}(D) \cap C^{0}(D)$.  
Let $k \in \{1, \ldots, n-1\}$ and $r \in \mathbb{R}$.  
Assume that $(dd^{c}u)^{k} = 0$ in $D$ and that $\overline{B}_{u}(r)$ is convex.  
Then any point of $S_{u}(r)$ lies in the relative interior of a complex $(n-k)$-dimensional convex set contained 
in $S_{u}(r)$.  
\end{theorem}
Here a complex $(n-k)$-dimensional convex set is a convex set whose affine hull is a complex $(n-k)$-dimensional affine space.  
We do not know whether there exists a complex $(n-k+1)$-dimensional locally closed submanifold without boundary through 
a point of $S_{u}(r)$ on which $u$ is pluriharmonic.  
However, the converse holds (see Proposition~\ref{proposition:b}).  

To prove Theorem~\ref{theorem:a}, we show the following minimum principle of a maximal plurisubharmonic function.  
\begin{theorem}\label{theorem:b}
Let $D \subset \mathbb{C}^{n}$ be a domain and let $u \in \mathrm{Psh}(D) \cap C^{0}(D)$ be a non-negative 
maximal plurisubharmonic function, that is, $u \geq 0$ and $(dd^{c}u)^{n} = 0$ in $D$.  
Assume that $u^{-1}(0) \subset D$ is non-empty and convex.  
Then any point of $u^{-1}(0)$ lies in the relative interior of a complex one-dimensional convex set contained in $u^{-1}(0)$.  
\end{theorem}
The condition that $u^{-1}(0)$ is convex can not be removed because of Sibony's example.  

Assume that $D \subset \mathbb{C}^{n}$ is a bounded convex domain.  
Lempert~\cite{Lem81} showed the relation between 
the Kobayashi distance and the pluricomplex Green function in $D$.  
Recall that the pluricomplex Green function of $D$ with logarithmic pole at $x \in D$ is defined as 
\[
K_{D, x}(z) = \sup\{u(z) ; u \in \mathrm{Psh}(D), u < 0, \limsup _{w \to x} u(w) - \log \|w-x\| < \infty \}  
\]
and $K_{D, x}$ is the unique solution for the following homogeneous Monge-Amp\`ere equation: 
\[
\begin{cases}
u \in \mathrm{Psh}(D) \cap L^{\infty}_{loc}(\overline{D} \setminus \{x\}), \\
(dd^{c}u)^{n} \equiv 0 \quad \text{in} \quad D \setminus \{x\}, \\  
u(z) - \log \|z-x\| = O(1) \quad \text{for} \quad z \to x, \\
\lim_{z \to p} u(z) = 0 \quad \text{for all} \quad p \in \partial D
\end{cases}
\]
(see \cite{Lem81}, \cite{Dem87}).  
Let $k_{D}$ be the Kobayashi distance on $D$.  
Lempert~\cite{Lem81} showed that 
\[
K_{D, x}(y) = \log \tanh k_{D}(x, y) \quad \text{for} \quad x, y \in D, x \neq y.  
\]  
Hence the balls centered at $x$ for the Kobayashi distance of $D$ coincide with the level sets of $K_{D, x}$.  
It was also shown in \cite{Lem81} that the balls for the Kobayashi distance of $D$ are convex.  
If $D$ is smoothly bounded and strictly convex, then 
$K_{D, x}$ is a smooth function in $D \setminus \{x\}$ such that $(dd^{c}K_{D, x})^{n-1} \neq 0$ 
in $D \setminus \{x\}$ and 
the balls centered at $x$ for the Kobayashi distance are strictly convex
(see \cite{Lem81} and \cite{Pat84}).  
However, if $D$ is a general bounded convex domain, 
balls for the Kobayashi distance are not always strictly (pseudo)convex and $K_{D, x}$ might be pluriharmonic in 
an open subset of $D$.  
In Section~\ref{section:3}, we show a relation between the supports of Monge-Amp\`ere currents 
$(dd^{c}K_{D, x})^{k}, (dd^{c}e^{K_{D, x}})^{k}$ and complex $k$-extreme points of balls centered at $x$ for the Kobayashi distance by using our results (Theorem~\ref{theorem:a} and Proposition~\ref{proposition:b}) and Lempert's results~\cite{Lem81} 
(see Section~\ref{section:3} for the definition of complex $k$-extreme points).  

\medskip

{\it Acknowledgment.}
The author would like to express his gratitude to Professor Junjiro Noguchi and Professor Katsutoshi Yamanoi 
for valuable advices and warm encouragements.  
Their suggestions improve the exposition greatly.  
The author also thanks Professor Shin-ichi Matsumura for valuable discussions.  
The author is supported by the Grant-in-Aid for Scientific Research (KAKENHI No. 25-902) 
and the Grant-in-Aid for JSPS fellows.  

\section{Proof of Theorem~\ref{theorem:a} and Theorem~\ref{theorem:b}}

Before giving the proof, we recall a notion of Jensen measure.  
The reader will find a much more study about Jensen measure in \cite{Ceg88}, \cite{Wik}.  
Let $\Omega \subset \mathbb{C}^{n}$ be a bounded domain and let $\mu$ be a positive, 
regular Borel measure on $\overline{\Omega}$.  
We say that $\mu$ is a Jensen measure with barycenter $z \in \overline{\Omega}$ for 
continuous plurisubharmonic functions, if $u(z) \leq \int_{\overline{\Omega}} u d\mu$ 
for every function $u \in \mathrm{Psh}(\Omega) \cap C^{0}(\overline{\Omega})$.  
We denote by $\mathcal{J}_{z}^{c}$ the set of Jensen measures for continuous 
plurisubharmonic functions having barycenter $z$.  
The following theorem is a consequence of Edwards' theorem.  
\begin{theorem}[\cite{Edw}]\label{theorem:2}
Let $\Omega \subset \mathbb{C}^{n}$ be a bounded domain, and let $v$ be a real valued 
lower semicontinuous function on $\overline{\Omega}$.  Then, for every $z \in \overline{\Omega}$, 
\[
\sup\{u(z); u \in \mathrm{Psh}(\Omega) \cap C^{0}(\overline{\Omega}),\, u \leq v 
\; \text{in} \; \; \overline{\Omega}\} 
= \inf \left\{\int_{\overline{\Omega}} v d\mu; \mu \in \mathcal{J}_{z}^{c}\right\}.  
\]
\end{theorem}

We first prove the following lemma.  

\begin{lemma}\label{lemma:2}
Let $D$ be a domain in $\mathbb{C}^{n}$.  
Let $u \in \mathrm{Psh}(D) \cap C^{0}(D)$.  
Then $(dd^{c}u)^{n} = 0$ in $D$ if and only if, 
for any $z \in D$ and any connected open neighborhood $U \subset \subset D$ of $z$, there exists 
a Jensen measure $\mu \in \mathcal{J}_{z}^{c}$ on $\overline{U}$ such that 
$u(z) = \int u d\mu$ and $\mu \neq \delta_{z}$ where $\delta_{z}$ is the Dirac mass at $z$.  
\end{lemma}
\begin{proof}
Let $x \in D$ and take the Euclidean open ball $B(x) \subset \subset D$ of a small radius centered at $x$.  
Let 
\[
v(z) = \sup\{\phi(z); \phi \in \mathrm{Psh}(B(x)) \cap C^{0}(\overline{B(x)}), \phi 
\leq u \; \text{on} \; 
\partial \overline{B(x)}\}
\] 
for $z \in \overline{B(x)}$.  
Then $v \in \mathrm{Psh}(B(x)) \cap C^{0}(\overline{B(x)})$ and $v$ is the unique solution of the Monge-Amp\`ere equation $(dd^{c}v)^{n} = 0$ in $B(x)$ such that 
$v = u$ on $\partial \overline{B(x)}$ (see \cite{BT76}).  

Assume that $(dd^{c}u)^{n} = 0$ in $D$.  
Then $u = v$ in $\overline{B(x)}$.  
Take the continuous function $h$ in $\overline{B(x)}$ which is harmonic in $B(x)$ and $h = u$ on $\partial \overline{B(x)}$.  
Then 
\[
u(z) = v(z) = \sup \{\phi(z); \phi \in \mathrm{Psh}(B(x)) \cap C^{0}(\overline{B(x)}), \phi \leq h \; \text{in} 
\; \overline{B(x)}\}.  
\]
By Theorem~\ref{theorem:2}, 
\[
u(z) =  \inf \left\{\int_{\overline{B(x)}} h d\mu; \mu \in \mathcal{J}_{z}^{c}\right\}
\]
where $\mathcal{J}_{z}^{c}$ is the set of Jensen measures for continuous plurisubharmonic functions having 
barycenter $z$ in $\overline{B(x)}$.  
For any $z$, $\mathcal{J}_{z}^{c}$ is a compact set in the weak topology and 
there exists $\mu \in \mathcal{J}_{x}^{c}$ such that $u(x) = \int_{\overline{B(x)}}h d\mu$.  
Since $u(x) \leq \int_{\overline{B(x)}}u d\mu \leq \int_{\overline{B(x)}}h d\mu$, 
we have $u(x) = \int_{\overline{B(x)}}u d\mu$.  
If $\mu = \delta_{x}$, $u(x) = h(x)$ and $u = h$ in $\overline{B(x)}$ by the strong maximum principle of a subharmonic function 
$u - h$.  
Then $u$ is harmonic in $B(x)$.  
Let $\tau$ be the probability measure invariant under the rotation on $\partial\overline{B(x)}$.  
Then $\tau \in \mathcal{J}_{x}^{c}$ and $u(x) = \int u d\tau$.   
We complete the proof of the necessity.    

We next show the sufficiency.  
Assume that $u \not \equiv v$ in $B(x)$.  
Then 
\[
m = \sup_{z \in \overline{B(x)}}\{v(z)- u(z)\}
\]
is positive.  
Let $A = \{z \in \overline{B(x)}; v(z) -u(z) = m\}$.  
Let $d$ be the Euclidean distance on $\mathbb{C}^{n}$
and let $s = d(A, \partial \overline{B(x)}) = \inf \{d(z, w); z \in A, w \in \partial \overline{B(x)}\}$.  
Since $u = v$ on $\partial \overline{B(x)}$, $s> 0$.  
Take $y \in A$ such that $d(y, \partial \overline{B(x)}) = s$.  
By the hypothesis, there exists a Jensen measure $\mu \in \mathcal{J}_{y}^{c}$ on $\overline{B(x)}$ such that 
$u(y) = \int u d\mu$ and $\mu \neq \delta_{y}$.  
Then 
$m \geq \int (v-u) d\mu \geq v(y) - u(y) = m$.  
Hence the support of $\mu$ is  contained in $A$.  
By Proposition~4.4 of \cite{DS}, 
there exists a positive current $T \neq 0$ of bidimension ($1, 1$) with compact support such that 
$dd^{c}T = \mu - \delta_{y}$.  
We define $\psi(z) = |z_{1}-x_{1}|^{2} + \cdots + |z_{n}-x_{n}|^{2}$.  
It follows that 
\[
0 < \langle dd^{c}T, \psi \rangle = \int \psi d\mu - \psi(y).  
\]
However, the right hand side of the above equation is non-positive by the choice of $y$ 
and this is a contradiction.  
Therefore $u = v$ in $B(x)$ and $(dd^{c}u)^{n} = 0$.  
\end{proof}

\begin{proof}[Proof of Theorem~\ref{theorem:b}]
Let $x \in u^{-1}(0)$.  
By Lemma~\ref{lemma:2}, 
there exists $\mu \in \mathcal{J}_{x}^{c}$ on a small neighborhood of $x$ such that $u(x) = \int u d\mu$ and $\mu \neq \delta_{x}$.  
Since $u \geq u(x) = 0$, 
we have $u = 0$ on $|\mu|$ where $|\mu|$ is the support of $\mu$ and $|\mu| \subset u^{-1}(0)$.  
Let $\mathrm{ch}(|\mu|)$ be the convex hull of $|\mu|$.  
Since $u^{-1}(0)$ is convex, 
$\mathrm{ch}(|\mu|) \subset u^{-1}(0)$.  
By Proposition~4.4 of \cite{DS}, 
there exists a positive current $T \neq 0$ of bidimension ($1, 1$) with compact support in $D$ such that 
$dd^{c} T = \mu - \delta_{x}$.  

Assume that $x$ does not lie in the relative interior of $\mathrm{ch}(|\mu|)$.  
The Hahn-Banach Theorem implies that there exist a linear function $L: \mathbb{C}^{n} \to \mathbb{R}$ 
and $t \in \mathbb{R}$ such that 
$L > t$ in the relative interior of $\mathrm{ch}(|\mu|)$ and $L(x) \leq t$.  
Then $0 = \langle dd^{c}T, L \rangle = \int L d\mu - L(x)> t- t = 0$ and this is a contradiction.  
Therefore $x$ lies in the relative interior of $\mathrm{ch}(|\mu|)$.  

We show that $\mathrm{ch}(|\mu|)$ contains a complex one-dimensional convex set. 
Assume that $\mathrm{ch}(|\mu|)$ is contained in a totally real convex set.  
There exists a holomorphic affine coordinate $(w_{1}, \ldots, w_{n})$ such that 
$\mathrm{ch}(|\mu|) \subset \{w \in \mathbb{C}^{n}; \mathrm{Re}\, w_{1} = \cdots = \mathrm{Re}\,w_{n} = 0\}$.   
Define a non-negative strictly plurisubharmonic function $g = (\mathrm{Re}\,w_{1})^{2} + \cdots + (\mathrm{Re}\,w_{k})^{2}$.  
(This is the special case of Lemma~1.2 of \cite{HW}).  
Then 
$0 < \langle T, dd^{c} g \rangle = \int g d\mu - g(x)$ since $T \neq 0$.  
However, the right hand side of the above equation is equal to $0$ (note that $x \in \mathrm{ch}(|\mu|)$).  
This is a contradiction.  
Hence $x$ lies in the relative interior of a complex one-dimensional convex set contained in $u^{-1}(0)$.  
\end{proof}

\begin{remark}
In the above proof, 
let $|T|$ be the support of $T$ and let $\mathrm{ch}(|T|)$ be the convex hull of $|T|$.  
Then $\mathrm{ch}(|\mu|) = \mathrm{ch}(|T|)$ since 
$|T| \subset \widehat{|\mu|}_{D} = \{z \in D; |f(z)| \leq \sup_{|\mu|}|f| \, \, \text{for any} \, \, f \in \mathcal{O}(D)\}$
(see \cite{Hor}, Proposition~4.3 of \cite{DS}) 
and 
\[
\mathrm{ch}(|\mu|) \subset \mathrm{ch}(|T|) \subset \mathrm{ch}(\widehat{|\mu|}_{D}) = 
\mathrm{ch}(|\mu|).  
\]
Assume that $u \in C^{2}(D)$.  
Then $\langle T, dd^{c}u \rangle = 0$.  
Conversely, let $S \neq 0$ be a positive current of bidimension $(1, 1)$ with compact support in $D$ such that 
$dd^{c}S + \delta_{x}$ is a positive measure and $\langle S, dd^{c}u \rangle = 0$.  
Then the convex hull of the support of $S$ contains the required complex one-dimensional convex set.  
\end{remark}

We need the following lemma to prove Theorem~\ref{theorem:a}.  

\begin{lemma}\label{lemma:1}
Under the same hypotheses as in Theorem~\ref{theorem:a} let $x \in S_{u}(r)$ which is not contained in 
the interior of $\overline{B}_{u}(r)$.  
Let $M \subset \mathbb{C}^{n}$ be a complex $k$-dimensional affine subspace through $x$ 
such that $M$ is tangent to $\overline{B}_{u}(r)$, 
that is, the intersection of $M$ and the interior of $\overline{B}_{u}(r)$ is empty.  
Then $x$ lies in the relative interior of a complex one-dimensional convex set contained in 
$S_{u}(r) \cap M$.  
\end{lemma}
Note that the above $M$ always exists since $\overline{B}_{u}(r)$ is convex.  
\begin{proof}
By taking a suitable holomorphic affine coordinate $(z_{1}, \ldots, z_{n})$, 
we may assume that $x = (0, \ldots, 0)$ and 
$M$ is defined by $z_{k+1} = \cdots = z_{n} =0$.  
Let $z' = (z_{1}, \ldots, z_{k})$ and $z'' = (z_{k+1}, \ldots, z_{n})$.  
Take open balls $0 \in U \subset \mathbb{C}^{k}$ and 
$0 \in V \subset \mathbb{C}^{n-k}$ of small radii such that $U \times V$ is contained in $D$.  
We show that $(dd^{c}u|_{M \cap D})^{k} = 0$ in $U \times \{0\}$.  
Let $a \in V$ and let $M(a)$ be the intersection of $D$ and the complex $k$-dimensional affine subspace defined by $z'' = a$.  
Let $\phi$ be a non-negative test function in $U$ and $\psi$ be a non-negative test function 
in $V$ such that $\psi(0) > 0$.  
We take a decreasing sequence of smooth plurisubharmonic functions $\{u_{i}\}$ on a 
neighborhood of the support of $\phi(z')\psi(z'')$ in $U \times V$ such that $u_{i}$ converges to $u$.  
Then 
$(dd^{c}u_{i})^{k} \to (dd^{c}u)^{k}$ $(i \to \infty)$ in the weak topology on the space of 
currents (see \cite{BT82}).  
By Fubini's theorem, we have that 
\begin{align*}
0  =  & \int_{U \times V} \phi(z')\psi(z'') (dd^{c}u)^{k} \bigwedge_{j = k+1}^{n}
\frac{\sqrt{-1}}{2}dz_{j}\wedge d\bar{z}_{j} \\
= & \lim_{i \to \infty} \int_{U \times V} 
\phi(z')\psi(z'') (dd^{c}u_{i})^{k} \bigwedge_{j = k+1}^{n}
\frac{\sqrt{-1}}{2}dz_{j}\wedge d\bar{z}_{j} \\
= & \lim_{i \to \infty} \int_{V} \psi(z'') 
\bigwedge_{j = k+1}^{n} \frac{\sqrt{-1}}{2}dz_{j}\wedge d\bar{z}_{j}
\int_{M(z'')} \phi(z') (dd^{c}u_{i}|_{M(z'')})^{k}.  
\end{align*}
Since $u$ is continuous in $D$, 
$u$ is uniformly continuous in a neighborhood of the support of $\phi(z')\psi(z'')$.  
If we consider $(dd^{c}u|_{M(z'')})^{k}$ as a current in $U$, 
$(dd^{c}u|_{M(z'')})^{k} \to (dd^{c}u|_{M(z''_{0})})^{k}$ when $z'' \to z''_{0} \in V$ 
in the weak topology (see Chapter~III of \cite{Dem}).  
Therefore $\int_{M(z'')} \phi(z') (dd^{c}u|_{M(z'')})^{k}$ is a non-negative continuous function of $z''$ and, 
by Lebesgue's dominated convergence theorem, we have that 
\[
\int_{V} \psi(z'') 
\bigwedge_{j = k+1}^{n} \frac{\sqrt{-1}}{2}dz_{j}\wedge d\bar{z}_{j}
\int_{M(z'')} \phi(z') (dd^{c}u|_{M(z'')})^{k} = 0.  
\]
Since $\psi(0) > 0$, $\int_{M \cap D} \phi(z') (dd^{c}u|_{M \cap D})^{k} = 0$ for any non-negative test function $\phi$.  
Hence $(dd^{c}u|_{M \cap D})^{k} = 0$ in $U \times \{0\}$.  

Since $\overline{B}_{u}(r)$ is convex and $M$ is tangent to $\overline{B}_{u}(r)$, $u|_{M \cap D}$ attains its minimum $r$ at $x$ and $u|_{M \cap D}^{-1}(r)$ is convex.  
By Theorem~\ref{theorem:b}, there exists a complex one-dimensional convex set in $S_{u}(r) \cap M$ which contains $x$ in its relative interior.  
This completes the proof.  
\end{proof}

\begin{proof}[Proof of Theorem~\ref{theorem:a}]
We prove Theorem~\ref{theorem:a} 
by induction on $l = n-k$ ($l = 1, \ldots, n-1$).  
Let $x \in S_{u}(r)$.  
If $x$ lies in the interior of $\overline{B}_{u}(r)$, the theorem holds.  
Hence we may assume that $x$ is not contained in the interior of $\overline{B}_{u}(r)$.  
If $l = 1$, then $k = n-1$ and the statement follows by Lemma~\ref{lemma:1}.  
For $l > 1$, the inductive hypothesis implies that there exists 
a complex $(n-k-1)$-dimensional convex set $C_{1}$ in  $S_{u}(r)$ 
which contains $x$ in its relative interior.  
Let $H$ be a complex affine hyperplane through $x$ 
and which is tangent to $\overline{B}_{u}(r)$.  
Take a complex $k$-dimensional affine subspace $M$ of $H$ such that $x \in M$ and 
$M \cap C_{1} = \{x\}$.  
By Lemma~\ref{lemma:1}, there exists a complex one-dimensional convex set 
$C_{2} \subset S_{u}(r)$ which contains 
$x$ in its relative interior.  
Then the convex hull $\mathrm{ch}(C_{1} \cup C_{2})$ is contained in $\overline{B}_{u}(r)$ since $\overline{B}_{u}(r)$ is convex and $\mathrm{ch}(C_{1} \cup C_{2})$ is contained in $S_{u}(r)$ by the maximum principle.  
\end{proof}

We now show the following proposition announced in Section~1.  
\begin{proposition}\label{proposition:b}
Let $D$ be a domain in $\mathbb{C}^{n}$ and $u \in \mathrm{Psh}(D) \cap C^{0}(D)$.  
Let $k \in \{1, \ldots, n\}$.  
Assume that any point of $D$ lies in a complex $(n-k+1)$-dimensional locally closed submanifold of $D$ 
without boundary on which $u$ is pluriharmonic.  
Then $(dd^{c}u)^{k} = 0$ in $D$.  
\end{proposition}
\begin{proof} 
We show 
\[
\int_{D} \phi(z) (dd^{c}u)^{k} \bigwedge_{i = 1}^{n-k} \frac{\sqrt{-1}}{2}dz_{j_{i}} \wedge d\bar{z}_{j_{i}}= 0 
\]
for any non-negative test function $\phi$ in $D$ and any $1 \leq j_{1} < \cdots < j_{n-k} \leq n$.  
Without loss of generality, we may assume that $j_{1} = k+1, \ldots, j_{n-k}=n$.  
Let $U \subset \mathbb{C}^{k}, V \subset \mathbb{C}^{n-k}$ be open subsets such that 
$U \times V \subset D$.  
We may assume that $D = U \times V$.  
Let $z' = (z_{1}, \ldots, z_{k})$ and $z'' = (z_{k+1}, \ldots, z_{n})$ and let 
$p: \mathbb{C}^{n} \to \mathbb{C}^{n-k}$ be the projection map such that $p(z) = z''$.  
Put $U(z'') = (U \times V) \cap p^{-1}(z'')$ for $z'' \in V$.  
Then, by the same argument as in the proof of Lemma~\ref{lemma:1}, 
it follows that 
\begin{align*}
& \int_{U \times V} \phi(z) (dd^{c}u)^{k} \bigwedge_{i = k+1}^{n} \frac{\sqrt{-1}}{2}dz_{i} \wedge d\bar{z}_{i} \\
= & \int_{V} \bigwedge_{i = k+1}^{n}\frac{\sqrt{-1}}{2}dz_{i} \wedge d\bar{z}_{i} 
\int_{U(z'')} \phi(z', z'')|_{U(z'')} (dd^{c}u|_{U(z'')})^{k}.  
\end{align*}
Hence it is enough to show that $(dd^{c}u|_{U(z'')})^{k} = 0$ in $U(z'')$ for any 
$z'' \in V$.  
Let $x \in U(z'')$.  
Then $x$ lies in a complex ($n-k+1$)-dimensional locally closed submanifold $N \subset D$ without boundary on which 
$u$ is pluriharmonic.  
It holds that 
$N \cap U(z'')$ contains an irreducible complex curve $R$ through $x$.  
Let $\pi:\widetilde{R} \to R$ be a resolution of singularities of $R$.   
Then $\pi^{*}u$ is harmonic in $\widetilde{R}$.  
Let $\tilde{x} \in \widetilde{R}$ such that $\pi(\tilde{x}) = x$.  
Let $\varphi$ be a continuous function from the closed unit disk $\overline{\Delta}$ in $\mathbb{C}$ to a small neighborhood of $\tilde{x}$ in $\widetilde{R}$ such that 
$\varphi$ is holomorphic in the unit disk $\Delta$ and $\varphi(0) = \tilde{x}$.  
Let $\lambda$ be the probability measure invariant under rotations on $\partial \Delta$.  
Then $\mu = \pi_{*} \varphi_{*} \lambda$ is a Jensen measure with barycenter $x$ on $U(z'')$ such that 
$\int u d\mu = u(x)$.  
By Lemma~\ref{lemma:2}, 
we have $(dd^{c}u|_{U(z'')})^{k} = 0$ in $U(z'')$. 
\end{proof}

\begin{corollary}\label{corollary:1}
Let $\Omega \subset \mathbb{C}^{n}$ be a tube domain and let $\omega \subset \mathbb{R}^{n}$ be a domain 
such that $\Omega = \{z \in \mathbb{C}^{n}; \mathrm{Re}\, z \in \omega\}$.  
Let $u \in \mathrm{Psh}(\Omega) \cap C^{0}(\Omega)$ such that $u(z)$ is independent of $\mathrm{Im}\, z$ and let $\tilde{u} \in C^{0}(\omega)$ such that $\tilde{u}(\mathrm{Re}\, z) = u(z)$.  
Let $k \in \{1, \ldots, n\}$.  
\begin{itemize}
\item[(1)]
Assume that $(dd^{c}u)^{k} = 0$ in $\Omega$.  
Then, for any point of $x \in \omega$, there exists a real $(n-k)$-dimensional convex set $C$ in $\omega$ such that 
the relative interior of $C$ contains $x$ and $\tilde{u}$ is constant on $C$.  
\item[(2)]
Assume that any point of $\omega$ lies in the relative interior of a real $(n-k+1)$-dimensional convex set on which $\tilde{u}$ is affine.  
Then $(dd^{c}u)^{k} = 0$ in $\Omega$.  
\end{itemize}
\end{corollary}
\begin{proof}
Since the statements are local, we may assume that $\Omega$ and $\omega$ are convex.  
By the hypothesis, $u$ and $\tilde{u}$ are convex functions.  
Hence $(1)$ and $(2)$ follow immediately from Theorem~\ref{theorem:a} and Proposition~\ref{proposition:b}.  
\end{proof}

Let $D$ be a Reinhardt domain in $(\mathbb{C}^{*})^{n}$.  
Then we have statements analogous to Corollary~\ref{corollary:1} for $u \in \mathrm{Psh}(D) \cap C^{0}(D)$
such that $u(z)$ is independent of $\mathrm{arg}\, (z)$ for $z \in D$.  

\begin{example}
Let $D \subset \mathbb{C}^{n}$ be a Reinhardt domain and $u \in \mathrm{Psh}(D) \cap C^{0}(D)$.  
Define 
\[
U(z_{1}, \ldots, z_{n}) = \sup_{0 \leq \theta_{1}, \ldots, \theta_{n} \leq 2\pi} u(e^{\sqrt{-1}\theta_{1}}z_{1}, \ldots, e^{\sqrt{-1}\theta_{n}}z_{n})
\]
for $z \in D$.   
Then $U \in \mathrm{Psh}(D) \cap C^{0}(D)$ and $U(z)$ is independent of $\mathrm{arg}\, (z)$.  
Assume that $(dd^{c}U)^{k} = 0$ in $D$.  
Then, for any $x \in D$ such that $x_{i} \neq 0$ for $1 \leq i \leq n$, there exists a real $(n-k+1)$-dimensional logarithmically convex set $C  \subset \{(|z_{1}|, \ldots, |z_{n}|) \in \mathbb{R}^{n}; z \in D \}$ such that the relative interior 
of $C$ contains $(|x_{1}|, \ldots, |x_{n}|)$ and $U$ is constant on $\{(z_{1}, \ldots, z_{n}) \in D; (|z_{1}|, \ldots, |z_{n}|) \in C\}$.   
\end{example}

\section{Balls for the Kobayashi distance in a bounded convex domain}\label{section:3}

Throughout this section $D$ will denote a bounded convex domain in $\mathbb{C}^{n}$.  
We define the complex $k$-extreme point of a convex set.  
\begin{definition}
Let $C \subset \mathbb{C}^{n}$ be a convex set.  
A point $z \in C$ is complex $k$-extreme point ($0 \leq k \leq n$) if $z$ lies in the relative interior of a 
complex $k$-dimensional convex set within $C$, but not a complex $(k+1)$-dimensional convex set 
within $C$.  
We denote by $E^{k}(C)$ the set of complex $k$-extreme points of $C$.  
\end{definition}

Let $\overline{B}_{D}(x, r) \subset D$ denote the closed ball of radius $r \geq 0$ centered at $x$ for 
the Kobayashi distance of $D$.  
Recall that $\overline{B}_{D}(x, r)$ is convex.  
In this section, we show a relation between the supports of Monge-Amp\`ere currents 
$(dd^{c}K_{D, x})^{k}$, $(dd^{c}e^{K_{D, x}})^{k}$ and $E^{k}(\overline{B}_{D}(x, r))$.  
We define 
\[
E_{x}^{<k} = \bigcup_{i = 0}^{k-1}\bigcup_{r \geq 0} E^{i}(\overline{B}_{D}(x, r)) \quad \text{for} 
\quad k = 1, \ldots, n.  
\]
For example, 
if $D = \{z \in \mathbb{C}^{n}; |z_{1}|<1, \ldots, |z_{n}|<1\}$ is a polydisk, 
$E_{0}^{< k}$ is the set of points $z \in D$ 
where there exists positive integers $1 \leq i(1) < \cdots < i(n-k+1) \leq n$ 
such that $|z_{i(1)}| = \cdots = |z_{i(n-k+1)}| = \max\{|z_{1}|, \ldots, |z_{n}|\}$.  
Let $k_{D}$ be the Kobayashi distance on $D$.  
We note that 
\begin{align*}
D \setminus E_{x}^{< k} & = \{ z \in D ; \text{$z$ lies in the relative interior of a 
complex $k$-dimensional}  \\
& \qquad 
\text{convex set contained in $\overline{B}_{D}(x, k_{D}(x, z))$}\}.  
\end{align*}  

\begin{remark}\label{remark:1}
The complex $k$-dimensional convex set appearing in the right hand side of the above equation is contained in the Kobayashi sphere of 
radius $k_{D}(x, z)$ centered at $x$ 
by the maximum principle.  
\end{remark}

Let $T$ be a current in $D$.  
We denote by $|T|$ the support of $T$.  
By results of \cite{BT82} and \cite{Dem87}, 
currents $(dd^{c}K_{D, x})^{k}, (dd^{c}e^{K_{D, x}})^{k}$ ($1 \leq k \leq n$) are well-defined in $D$.  

\begin{theorem}\label{theorem:1}
Let $D \subset \mathbb{C}^{n}$ be a bounded convex domain.  
Then: 
\begin{itemize}
\item[(i)]
$|(dd^{c}e^{K_{D, x}})^{n-k+1}| \subset \overline{E_{x}^{< k}} \cap D \subset 
|(dd^{c}K_{D, x})^{n-k}|$ 
for $1 \leq k \leq n-2$.  
\item[(ii)]
$|dd^{c}K_{D, x}| = \overline{E_{x}^{<(n-1)}} \cap D$.  
\end{itemize}
\end{theorem}

\begin{remark}\label{remark:2}
A current $dK_{D, x} \wedge d^{c}K_{D, x} \wedge (dd^{c}K_{D, x})^{k}$  ($1 \leq k \leq n-1$) is well-defined in $D \setminus\{x\}$.  
Since 
\[
(dd^{c}e^{K_{D, x}})^{k} = e^{k K_{D, x}}(dd^{c}K_{D, x})^{k} + k e^{k K_{D, x}}dK_{D, x} \wedge d^{c}K_{D, x} \wedge 
(dd^{c}K_{D, x})^{k-1},  
\] 
it follows $|(dd^{c}e^{K_{D, x}})^{k}| = |(dd^{c}K_{D, x})^{k}| \cup |dK_{D, x} \wedge d^{c}K_{D, x} \wedge 
(dd^{c}K_{D, x})^{k-1}|$ in $D \setminus \{x\}$.  
\end{remark}

\begin{proof}
By Theorem~\ref{theorem:a}, it follows that 
$D \setminus |(dd^{c}K_{D, x})^{n-k}| \subset D \setminus \overline{E_{x}^{<k}}$ for $1 \leq k \leq n-1$.  
For $z \in D \setminus \overline{E_{x}^{<k}}$, there exists a complex $k$-dimensional convex set which contains $z$ 
in its relative interior on which $e^{K_{D, x}}$ is constant.  
Then $D \setminus \overline{E_{x}^{< k}} \subset D \setminus |(dd^{c}e^{K_{D, x}})^{n-k+1}|$ for $1 \leq k \leq n-2$ by Proposition~\ref{proposition:b}.  
This completes the proof of (i).  

Next, we prove (ii).  
Since we have already proved that $\overline{E_{x}^{<(n-1)}} \cap D \subset |dd^{c}K_{D, x}|$, 
it is enough to show $D \setminus \overline{E_{x}^{<(n-1)}} \subset D \setminus |dd^{c}K_{D, x}|$.  

Before giving the proof, we recall that a complex geodesic (for the Kobayashi distance) of a convex 
domain $D$ 
is a holomorphic map $\varphi$ from the unit disk $\Delta \subset \mathbb{C}$ into $D$ such that 
$k_{\Delta}(\zeta_{1}, \zeta_{2}) = k_{D}(\varphi(\zeta_{1}), \varphi(\zeta_{2}))$ 
for every pair of points $\zeta_{1}, \zeta_{2} \in \Delta$.  
For any two points in a convex domain $D$, there exists a complex geodesic which through 
those (see \cite{Lem81} and \cite{RW83}).  

Let $y \in D \setminus \overline{E_{x}^{<(n-1)}}$.  
Let $\varphi: \Delta \to D$ be a complex geodesic which through 
$x$ and $y$ such that $\varphi(0) = x$.  
Then there exists Lempert's projection $\rho$, that is, 
a holomorphic retraction $\rho: D \to \varphi(\Delta) \subset D$ such that 
$\rho$ is the identity map on $\varphi(\Delta)$ and 
$\rho^{-1}(z)$ is the intersection of $D$ and the complex affine hyperplane which is tangent to 
$\overline{B}_{D}(x, k_{D}(x, z))$ for $z \in \varphi(\Delta)$ (see \cite{Lem81} and \cite{RW83}).  
Let $W \subset D \setminus \overline{E_{x}^{<(n-1)}}$ be a small open neighborhood of $y$ such that 
$\rho^{-1}(z) \cap W$ is connected for any $z \in \varphi(\Delta) \cap W$.  
Let $S_{D}(x, t)$ be the Kobayashi sphere of radius $t > 0$ centered at x.  
We show that 
\[
\rho^{-1}(z) \cap W \subset S_{D}(x, k_{D}(x, z)) 
\]
for any $z \in \varphi(\Delta) \cap W$.  
Since $\rho^{-1}(z) \cap W \cap S_{D}(x, k_{D}(x, z))$ is 
non-empty and closed in $\rho^{-1}(z) \cap W$, 
it is enough to show that $\rho^{-1}(z) \cap W \cap S_{D}(x, k_{D}(x, z))$ is open in 
$\rho^{-1}(z) \cap W$.  
Let $\gamma \in \rho^{-1}(z) \cap W \cap S_{D}(x, k_{D}(x, z))$.  
We have $k_{D}(x, z) = k_{D}(x, \gamma)$.  
There exists a complex $(n-1)$-dimensional convex set 
$C \subset S_{D}(x, k_{D}(x, z))$ which contains $\gamma$ in its relative interior (see Remark~\ref{remark:1}).  
%Because $\rho^{-1}(z)$ is tangent to $\overline{B}_{D}(x, k_{D}(x, z))$, 
Note that $\varphi$ is a biholomorphic map from $\Delta$ to $\varphi(\Delta)$.  
By the distance decreasing property of the Kobayashi distance, 
$\varphi^{-1}(\rho(C)) \subset \{\zeta \in \mathbb{C}; 
|\zeta| \leq \tanh k_{D}(x, z)\}$.  
Since $|\varphi^{-1}(\rho(\gamma))| = \tanh k_{D}(x, z)$, 
the maximum principle implies 
that $\varphi^{-1}(\rho(C))$ is a point and $C$ is contained in 
$\rho^{-1}(z) \cap S_{D}(x, k_{D}(x, z))$.  
Hence $\rho^{-1}(z) \cap W \cap S_{D}(x, k_{D}(x, z))$ is open in 
$\rho^{-1}(z) \cap W$ and $\rho^{-1}(z) \cap W \subset S_{D}(x, k_{D}(x, z))$.  
Then $K_{D, x}$ is constant on $\rho^{-1}(z) \cap W$.  
Therefore 
$K_{D, x}(y) = K_{D, x}(\rho(y))$ for $y \in W$.  
Since $K_{D, x}$ is harmonic on $\varphi(\Delta) \setminus \{x\}$, $K_{D, x} \circ \rho$ is pluriharmonic and 
$dd^{c}K_{D, x} = 0$ in $W$.  
This completes the proof of (ii).  
\end{proof}

\section{The case of a convex balanced domain}
Let $D \subset \mathbb{C}^{n}$ be a bounded convex balanced domain, that is, 
let $D$ be a bounded convex domain such that 
$\lambda z \in D$ for any $z \in D$ and $\lambda \in \mathbb{C}$, $|\lambda| \leq 1$.  
Let $\delta:\mathbb{C}^{n} \to [0, \infty)$ be the Minkowski function of $D$.  
Then it is known that $\log \delta \in \mathrm{Psh}(D)$ and $\log \delta$ is the 
pluricomplex Green function of $D$ with logarithmic pole at $\{0\}$.  

\begin{proposition}\label{proposition:1}
Let $D$ be a bounded convex balanced domain in $\mathbb{C}^{n}$.  
Then $\overline{E_{0}^{<k}} \cap D = |(dd^{c} \log \delta)^{n-k}|$ for $1 \leq k \leq n-1$.  
\end{proposition}
\begin{proof}
By Theorem~\ref{theorem:1} and Remark~\ref{remark:2}, 
it follows that 
\[
|d\log \delta \wedge d^{c} \log \delta \wedge (dd^{c}\log \delta)^{n-k}| 
\subset \overline{E_{0}^{<k}} \cap D \subset |(dd^{c} \log \delta)^{n-k}|  
\]
for $1 \leq k \leq n-1$.  
By the same argument as that used on the proof of Proposition~2 of \cite{Tib}, 
$|(dd^{c} \log \delta)^{n-k}| = |d\log \delta \wedge d^{c} \log \delta \wedge (dd^{c}\log \delta)^{n-k}|$.  
This completes the proof.  
\end{proof}
\begin{remark}
We can also derive Proposition~\ref{proposition:1} from Proposition~\ref{proposition:b} and Theorem~\ref{theorem:1}.  
\end{remark}
Let $A(\overline{D})$ be the uniform algebra consists of all continuous functions on 
$\overline{D}$ which can be approximated uniformly on $\overline{D}$ 
by continuous functions which are 
holomorphic on $D$.  
Let $\partial_{S}\overline{D}$ denote the Shilov boundary of 
$\overline{D}$ for $A(\overline{D})$.  
Note that, if $D$ is a bounded convex domain, 
$A(\overline{D})$ is equal to the uniform algebra consists of all continuous functions 
on $\overline{D}$ which can be approximated uniformly on $\overline{D}$ by continuous functions 
which are holomorphic on a neighborhood of $\overline{D}$.  
If $D$ is a bounded convex balanced domain, 
$|(dd^{c} \log \delta)^{n-1}|$ is equal to 
$\bigcup_{0 \leq r < 1} r \partial_{S}\overline{D}$ by Proposition~2 of \cite{Tib}.  
Hence Proposition~\ref{proposition:1} implies the fact that, in a complex normed vector space with finite dimension, the closure of the set of the complex extreme points of 
the closed unit ball is equal to the Shilov boundary of it.  

\begin{example}
Let $P = \{z \in \mathbb{C}^{n}; |z_{1}| <1, \ldots, |z_{n}|<1\}$ be a polydisk and 
let $\delta (z) = \max\{|z_{1}|, \ldots, |z_{n}|\}$ be the Minkowski function of $P$.  
Then $|(dd^{c}\log \delta)^{k}|$ ($1 \leq i \leq n-1$) is the set of points $z \in P$ where there exist positive integers 
$1 \leq i(1) < \cdots < i(k+1) \leq n$ such that $|z_{i(1)}| = \cdots = |z_{i(k+1)}| = \max\{|z_{1}|, \ldots, |z_{n}|\}$.  
\end{example}

\begin{example}
Let $1 < n_{1} < n_{2}$ be positive integers and 
let $B_{j}$ be the Euclidean unit ball in $\mathbb{C}^{n_{j}}$ ($j = 1, 2$).  
Let $D = B_{1} \times B_{2}$ and 
let 
\[
\delta(z, w) = \max\left\{\sqrt{|z_{1}|^{2} + \cdots |z_{n_{1}}|^{2}}, \sqrt{|w_{1}|^{2} + \cdots |w_{n_{2}}|^{2}}\right\}
\quad (z \in B_{1}, w \in B_{2})
\]
be the Minkowski function of $D$.  
Then 
\begin{align*}
& |(dd^{c}\log\delta)| = \cdots = |(dd^{c}\log\delta)^{n_{1}-1}|  = D, \\
& |(dd^{c}\log\delta)^{n_{1}}| = \cdots = |(dd^{c}\log\delta)^{n_{2}-1}| = \{(z, w) \in B_{1} \times B_{2}; \|z\| \leq \|w\|\}, \quad \text{and} \\
& |(dd^{c}\log\delta)^{n_{2}}| = \cdots = |(dd^{c}\log\delta)^{n_{1} + n_{2} -1}| = \{(z, w) \in B_{1} \times B_{2}; \|z\| = \|w\|\}.  
\end{align*}
\end{example}

\par\noindent{\scshape \small
Department of Mathematics, \\
Tokyo Institute of Technology \\
Oh-Okayama, Meguro, Tokyo (Japan) }
\par\noindent{\ttfamily tiba.y.aa@m.titech.ac.jp}

\begin{thebibliography}{99}
\bibitem{BK}
E. Bedford and M. Kalka, 
Foliations and complex Monge-Amp\`ere equations, Communications on Pure and Applied Mathematics, {\bf 30}
(1977), 543--571.  
\bibitem{BT76}
E. Bedford and B. A. Taylor, 
The Dirichlet problem for a complex Monge-Amp\`ere equation, Invent. Math. {\bf 37} 
(1976), 1--44.  
\bibitem{BT82}
E. Bedford and B. A. Taylor, 
A new capacity for plurisubharmonic functions, Acta Math. {\bf 149} (1982), 1--40.  
%\bibitem{BT87}
%E. Bedford and B. A. Taylor, 
%Fine topology, \v{S}ilov boundary, and $(dd^{c})^{n}$, 
%J. Funct. Anal. {\bf 72} (1987), 225--251.  
\bibitem{Ceg88}
U. Cegrell, 
Capacities in Complex Analysis, Vieweg, Braunschweig, 1988.  
\bibitem{Dem87}
J. P. Demailly, 
Mesures de Monge-Amp\`ere et mesures pluriharmoniques, Math. Z. 
{\bf 194} (1987), 519--564.  
%\bibitem{Dem93}
%J. P. Demailly, 
%Monge-Amp\`ere operators, Lelong numbers and intersection theory, 
%Complex analysis and geometry, Univ. Ser. Math., Plenum, New York, (1993), 115--193
\bibitem{Dem}
J. P. Demailly, 
Complex analytic and differential geometry, OpenContent Book.  
Version of Thursday June 21, 2012.  Available at the authors web page.  
\bibitem{Duj}
R. Dujardin, 
Wermer examples and currents, Geom. Funct. Anal. {\bf 20} (2010), 398--415.  
\bibitem{DS}
J. Duval and N. Sibony, 
Polynomial convexity, rational convexity, and Currents, Duke Math. J. {\bf 79} (1995), 487--513.  
\bibitem{Edw}
D, Edwards, Choquet boundary theory for certain spaces of lower semicontinuous functions, in Function Algebras (Proc. Internat. Sypos. on Function Algebras, Tulane Univ., 1965) (Birtel, F., ed.), pp. 300--309, Scott-Foresman, Chicago, III., 1966.  
%\bibitem{Gam}
%T. W. Gamelin, Uniform algebras, Prentice Hall, 1969, Second Edition, Chelsea Press, 1984.  
%\bibitem{HR}
\bibitem{HW}
R. Harvey and R. O. Wells, 
Holomorphic approximation and hyperfunction theory on on a $C^{1}$ tottaly real submanifold 
of a complex manifold, Math. Ann. {bf 197} (1972), 287--318.  
\bibitem{Hor}
L. H\"ormander, 
An introduction to complex analysis in several variables, Third edition, North-Holland Mathematical Library, {\bf 7}. North-Holland Publishing Co., Amsterdam, (1990). 
\bibitem{Kli}
M. Klimek, 
Pluripotential Theory, London Mathematical Society Monographs {\bf 6}. Oxford
University Press, New York, (1991).
\bibitem{Lem81}
L. Lempert, 
La m\'etrique de Kobayashi et la repr\'esentation des domaines sur la boule, 
Bull. Soc. Math. France {\bf 109} (1981), 427--474.  
\bibitem{Pat84}
G. Patrizio, 
Parabolic exhaustions for strictly convex domains, Manuscripta Math. {\bf47} (1984), 271--309. 
\bibitem{RW83}
H. Royden and P.-M. Wong, 
Carath\'eodory and Kobayashi metric on convex domains, Preprint (1983).  
\bibitem{Tib}
Y. Tiba, 
Shilov boundaries of the pluricomplex Green function's level sets, 
preprint, UTMS 2013-12.  
%\bibitem{Ves82}
%E. Vesentini, 
%Complex geodesics and holomorphic maps, Symposia Mathematica {\bf 26} (1982), 211--230.  
\bibitem{Wik}
F. Wikstr\"{o}m,
Jensen measures and boundary values of plurisubharmonic functuions, 
Ark. Mat. {\bf 39} (2001), 181--200.  


\end{thebibliography}
\end{document}